\documentclass[12pt]{amsart}

\usepackage{amsfonts,latexsym,rawfonts,amsmath,amssymb,amsthm,mathrsfs}
\usepackage[plainpages=false]{hyperref}
\usepackage{graphicx}
\usepackage{comment}

\numberwithin{equation}{section}

\RequirePackage{color}
\theoremstyle{plain}

\newtheorem{coro}{Corollary}[section]

\newtheorem{lemma}{Lemma}[section]
\newtheorem{prop}{Proposition}[section]

\newtheorem{thm}{Theorem}[section]
\newcommand{\beq}{\begin{equation}}
\newcommand{\eeq}{\end{equation}}
\newcommand{\beqs}{\begin{eqnarray*}}
\newcommand{\eeqs}{\end{eqnarray*}}
\newcommand{\beqn}{\begin{eqnarray}}
\newcommand{\eeqn}{\end{eqnarray}}
\newcommand{\beqa}{\begin{array}}
\newcommand{\eeqa}{\end{array}}

\def\vol{\text{Vol}}

\let\a\alpha
\let\b\beta

\let\f\varphi

\let\l\lambda

\let\r\rho

\let\th\theta

\def\ii{\sqrt{-1}}

\def\nn{\bar{u}_{(x_{i,\n},\u_{i,\n})}}
\def\VV{dV_{\th}}
\def\VVV{dV_{\th_0}}
\def\qq{\frac{2}{n}}
\def\pp{\frac{\partial}{\partial t}}
\def\ww{2+\frac{2}{n}}
\def\ee{\frac{n+2}{2}}
\def\rr{\frac{n+2}{n}}
\def\tt{\int_M}
\def\ss{\triangle_b}

\def\aa{\tilde{r}}

\def\um{u_{\min}}
\def\umm{u_{\max}}
\def\nn{\nabla_{\th_0}}
\def\dd{\frac{d}{dt}}
\def\LL{_{L^2(M,\th_0)}}

\begin{document}
\title{The Exponential Convergence of the CR Yamabe Flow}

\author{Weimin Sheng}
\address{School of Mathematical Sciences, Zhejiang University, Hangzhou 310027, China.}
\email{weimins@zju.edu.cn}

\author{Kunbo Wang}
\address{School of Mathematical Sciences, Zhejiang University, Hangzhou 310027, China.\\
Current address: College of Sciences, China Jiliang University,  Hangzhou 310018, China.}
\email{21235005@zju.edu.cn}

\thanks{
The authors were supported  by NSF in China No. 11571304 .}



\begin{abstract}
In this paper, we study the CR Yamabe flow with zero CR Yamabe invariant. We use the CR Poincar\'e inequality and a Gagliardo-Nirenberg type interpolation inequality to show that this flow has long time solution and the solution  converges to a contact form with flat pseudo-Hermitian scalar curvature exponentially.
\end{abstract}

\maketitle

\baselineskip16pt
\parskip3pt

\section{Introduction}
Let $(M^n, g)$ be a smooth, compact Riemannian manifold without boundary, and its dimension $n\geq 3$. The Yamabe problem \cite{Ya} is to find a metric conformal to $g$ such that it has constant scalar curvature.  This problem was solved by Yamabe, Trudinger, Aubin and Schoen in \cite{Ya, Tr, A, S} .
A different approach to the Yamabe problem is the Yamabe flow, which was proposed by Hamilton \cite{H}. Denote $R_g$ the scalar curvature of $g$ and $r_g$ the mean value of $R_g$, i. e.
$$r_g=\frac{\tt R_g dV_g}{\tt dV_g}.$$
Consider the following parabolic equation
\begin {equation} \label{1.2}
\frac{\partial g}{\partial t}=-(R_g-r_g)g.
\end {equation}
Hamilton showed the short time existence for \eqref{1.2} in \cite{H}. Chow \cite{Chow} proved that \eqref{1.2} approaches to a metric of constant scalar curvature provided that the initial metric is locally conformally flat and has positive Ricci curvature. In \cite{Ye}, Ye obtained uniform a priori $C^1$ bounds for the solution of \eqref{1.2} on any conformally flat manifold, and showed that \eqref{1.2} smoothly converge to a metric of constant scalar curvature. Ye also proved that the Yamabe flow \eqref{1.2} exits for all time and converges smoothly to a unique limit of constant scalar curvature provided that the initial metric is scalar negative or scalar flat. By use of the general concentration-compactness result \cite{Struwe}, Schwetlick and Struwe\cite{SS} proved the convergence of the Yamabe flow when $3\leq n\leq 5$  provided that the initial metric has large energy. In \cite{B05},  Brendle proved the convergence of the flow for arbitrary  initial energy.

The CR geometry, which is the abstract model of real hypersurfaces in complex manifolds, has a lot of analogy with the geometry of Riemannian manifolds. Many mathematicians have made outstanding contributions in this field, such as Chern and Moser \cite{CM}, Fefferman\cite{Feff}, Folland \cite{F}, Folland and Stein\cite{FS}, Jerison and Lee \cite{JL-JDG, JL-JAMS, JL89}, Tanaka \cite{T}, and Webster\cite{W}, etc.. Jerison and Lee\cite{JL-JDG} studied a Yamabe type problem on CR manifolds. To distinguish it with the Riemannian Yamabe problem, it is called the CR Yamabe problem.  Suppose that $(M,\th)$ is a compact strongly psedo-convex CR manifold of real dimension $2n+1$ with a given contact form $\th$. The CR Yamabe problem is to find a contact form $\tilde{\th}$ conformal to $\th$ such that its Webster scalar curvature is constant.
If we define a new contact form $\tilde{\th}=u^{\frac{2}{n}}\th$, where $u>0$, and denote $\tilde{R}$  ($R$ resp. ) the pseudo-Hermitian Webster scalar curvature with respect to the contact form $\tilde{\th}$ ($\th$, resp.), then the CR Yamabe problem is reduced to solve the following CR Yamabe equation
\begin{equation}\label{CRYE}
-(2+\frac{2}{n})\triangle_b u+Ru=\tilde{R}u^{1+\frac{2}{n}},
\end{equation}
where $\triangle_b$ is the sub-Lapacian of $M$.
The CR Yamabe invariant is defined as
$$\l(M,\th)=\inf\{\frac{\int_M [(2+\frac{2}{n})\|\nabla_{\th} u\|^2+Ru^2]dV_{\th}}{(\int_M u^{2+\frac{2}{n}}dV_{\th})^{\frac{n}{n+1}}}:u>0,u\in S^2_1(M)\}.$$
Here $dV_{\th}$ is the volume form with respect to the contact form $\th$, $S^2_1(M)$ is the Folland-Stein space, which is the completion of $C^1(M)$ with respect to the norm
$$||u||_{S^2_1(M)}=(\int_M(|\nabla_{\th} u|^2_{\th}+|u|^2)dV_{\th})^{\frac{1}{2}}.$$
Jerison and Lee \cite{JL-JDG} solved the CR Yamabe problem when $n\geq 2$ and $M$ is not locally CR equivalent to the sphere. The remaining cases were solved by Gamara\cite{Ga}, and Gamara, Yacoub \cite{GY}.\\

Since $\l(M,\th)$ is determined by the CR structure, which is independent of the choice of $\th$, we denote it by $\l(M)$ from now on.  It is natural to ask if we can solve the CR Yamabe problem by a parabolic argument. Namely, as an analogue to the Yamabe flow on a Riemannian manifold, one can construct the CR Yamabe flow as follows:
\begin{equation}\label{CRYF}
\frac{\partial}{\partial t}\tilde{\th}(t)=-(\tilde{R}-\tilde{r})\tilde{\th}(t).
\end{equation}
Here $\tilde{r}$ is the average value of the pseudohermitian scalar curvature $\tilde{R}$, defined by
$$\tilde{r}=\frac{\int_M \tilde{R}dV_{\tilde{\th}}}{\int_M dV_{\tilde{\th}}}.$$
The CR Yamabe flow was firstly studied by Chang and Cheng \cite{CC}. They proved the short time existence in all dimensions and obtained a Harnack type inequality in dimension three. Zhang \cite{Z}  proved the long time existence and convergence for the case $\l(M)<0$. For the case $\l(M)>0$, Ho \cite{Ho12} proved the long
time existence for all dimensions, and the convergence when $M$ is the sphere. Ho and the authors \cite{HSW} proved the convergence when $n=1$ recently.

For a given contact form $\th_0$ on $M$, we say $\tilde{\th}$ is conformal to $\th_0$ if there is a positive function $f$ such that
$$\tilde{\th}=f\th_0.$$
Let $[\th_0]$ be the conformal class of a given contact form $\th_0$ on $M$. If we assume that $\l(M)=0$, then we can find a contact form $\th\in [\th_0]$ with flat pseudohermitian scalar curvature. Without loss of generalization, we may assume it is $\th_0$ itself. We consider the following CR Yamabe flow:
\begin{equation}\label{CRYF1}
\left\{
\begin{array}{ll}
\frac{\partial}{\partial t}\tilde{\th}(t)&=-(\tilde{R}-\tilde{r})\tilde{\th}(t),\\[0.2cm]
\tilde{\th}(t)&=u^{\frac{2}{n}}(t)\th_0,\\
\tilde{\th}(t)|_{t=0}&=\th.
\end{array}
\right.
\end{equation}
Here $\th$ may be $\th_0$ or some other fixed contact form from the conformal class $[\th_0]$, i.e.
$$\th=u(\cdot,0)^{\qq}\th_0.$$
In this paper, we follow the idea of Ye \cite{Ye}(Page 45-47) to prove the following main theorem:
\begin {thm}\label {main}
Let $(M,\th_0)$ be a smooth, strictly pseudo-convex $2n+1$ dimensional compact CR manifold. Suppose $\l(M)=0$, then the CR Yamabe flow \eqref {CRYF1} exists for all time, and converges to a contact form with flat pseudo-Hermitian scalar curvature exponentially.
\end {thm}

The convergence argument depends on a Poincar\'e inequality and  a CR Gagliardo-Nirenberg type inequality.  In section 2, we recall some basic concepts in CR geometry, derive a global version of Poincar\'e inequality on CR manifolds. In section 3, we prove the long time existence and exponential
 convergence of the CR Yamabe flow \eqref {CRYF1}. In the appendix, we prove a Gagliardo-Nirenberg type interpolation inequality in CR geometry.

\section{Preliminaries and Notations}
Let $M$ be an orientable, real, $(2n+1)$-dimensional manifold. A CR structure on $M$ is given by a complex $n$-dimensional
subbundle $T_{1,0}$ of the complexified tangent bundle ${\mathbb C}TM$ of $M$, satisfying $T_{1,0}\cap T_{0,1}=\{0\}$, where
$T_{0,1}=\bar{T}_{1,0}$. We assume the CR structure is formally integrable, that is, $T_{1,0}$ satisfies the Frobenius condition
$[T_{1,0}, T_{1,0}]\subset T_{1,0}$. Set $G=Re(T_{1,0}\oplus T_{0,1})$. Then $G$ is a real $2n$-dimensional sub-bundle of $TM$.
Then $G$ carries a natural complex structure map: $J: G\rightarrow G$ given by $J(V+\bar{V})=\ii(V-\bar{V})$ for $V\in T_{1,0}$.

Let $E\subset T^{\ast}M$ denote the real line bundle $G^{\bot}$. Because we assume $M$ is orientable, and the complex structure $J$ induces an
orientation on $G$, $E$ has a global non-vanishing section. A choice of such a 1-form $\th$ is called a pseudo-Hermitian structure on $M$. Associated with such $\theta$, the real symmetric bilinear form $L_\theta$ on $G$:
\begin{equation*}
L_\theta(V,W)=d\theta(V,JW),~~V,W\in G
\end{equation*}
is called the $Levi-form$ of $\theta$. $L_\theta$ extends by complex linearity to $\mathbb{C}G$, and induces a Hermitian form on $T_{1,0}$, which we write
\begin{equation*}
L_\theta(V,\bar W)=-\ii d\theta(V,\bar W),~~V,W\in T_{1,0}
\end{equation*}
If $\theta$ is replaced by $\tilde\theta=f\theta$, $L_\theta$ changes conformally by $L_{\tilde\theta}=fL_\theta$. We assume that $M$ is {\it{strictly pseudo-convex}}, that is, $L_\theta$ is positive definite for a suitable $\theta$. In this case, $\theta$ defines a contact structure on $M$, and we call $\theta$ a contact form. Then we define the volume form on $M$ as $dV_{\theta}=\theta\wedge (d\theta)^n$.

We can choose a unique $T$ called the characteristic direction such that $\theta(T)=1$, $d\theta(T, \cdot)=0$, and $TM=G\oplus \mathbb{R}T$. Then we can define a co-frame $\{\theta, \theta^1, \theta^2, \cdots, \theta^n\}$ satisfying $\theta^\alpha(T)=0$, which is called admissible coframe. Its dual frame $\{T, Z_1, Z_2, \cdots, Z_n\}$ is called admissible frame. In this co-frame, we have $d\theta=\ii h_{\alpha\bar\beta}\theta^\alpha\wedge\theta^{\bar\beta}$, $h_{\alpha\bar\beta}$ is a Hermitian matrix. $h_{\a\bar{\b}}$ and $h^{\a\bar{\b}}$ are used to lower and raise the indices.

The sub-Laplacian operator $\triangle_b$ is defined by
$$\int_M (\triangle_b u)fdV_{\th}=-\int_M\langle du,df\rangle_{\th}dV_{\th},$$
for all smooth function $f$. Here $<,>_{\th}$ is the inner product induced by $L_{\th}$. We denote $|\nabla_{\th}u|^2=\langle du,du\rangle_{\th}$. Tanaka \cite{T} and Webster \cite{W} showed there is a natural connection in the bundle $T_{1,0}$ adapted to a pseudo-Hermitian structure, which is called the Tanaka-Webster connection. To define this connection, we choose an admissible co-frame $\{\th^{\a}\}$ and dual frame $\{Z_{\a}\}$ for $T_{1,0}$. Then there are uniquely determined 1-forms $\omega_{\a\bar{\b}}$, $\tau_{\a}$ on $M$, satisfying
\begin{eqnarray}
d\theta^\alpha &=& \omega^\alpha_\beta\wedge\theta^\beta+\theta\wedge\tau^\alpha,\\
dh_{\alpha\bar\beta} &=& h_{\alpha\bar\gamma}\omega_{\bar\beta}^{\bar\gamma}+\omega_{\alpha}^{\gamma}h_{\gamma\bar\b},\\
\tau_\alpha\wedge\theta^\alpha &=& 0.
\end{eqnarray}
From the third equation, we can find $A_{\alpha\gamma}$, such that
$$\tau_\alpha=A_{\alpha\gamma}\theta^\gamma$$
and $A_{\alpha\gamma}=A_{\gamma\alpha}$. Here $A_{\alpha\gamma}$ is called the pseudohermitian torsion.
With this connection, the covariant differentiation is defined by
$$
\nabla Z_\alpha=\omega_\alpha^\beta\otimes Z_\beta,~~~~\nabla Z_{\bar\alpha}=\omega_{\bar\alpha}^{\bar\beta}\otimes Z_{\bar\beta},~~~~\nabla T=0.
$$
$\{\omega^{\alpha}_\beta\}$ are called connection 1-forms.
For a smooth function $f$ on $M$, we write $f_\alpha=Z_\alpha f,~~f_{\bar\alpha}=Z_{\bar\alpha} f,~~f_0=Tf$, so that
$df=f_\alpha \theta_\alpha+f_{\bar\alpha} \theta_{\bar\alpha}+f_0 \theta$. The second covariant differential $\nabla^2 f$ is the 2-tensor with components
\begin{equation*}
\begin{split}
f_{\alpha\beta} &=\overline{\bar f_{\bar\alpha\bar\beta}}=Z_\beta Z_\alpha f-\omega_\alpha^\gamma(Z_\beta) Z_\gamma f, ~~f_{\alpha\bar\beta} =\overline{\bar f_{\bar\alpha\beta}}=Z_{\bar\beta} Z_\alpha f-\omega_\alpha^\gamma(Z_{\bar\beta}) Z_\gamma f,\\
f_{0\alpha} &=\overline{\bar f_{0\bar\alpha}}=Z_\alpha Tf,~~f_{\alpha0}=\overline{\bar f_{\bar\alpha 0}}=TZ_\alpha f-\omega_\alpha^\gamma(T) Z_\gamma f,~~f_{00}=T^2 f.
\end{split}
\end{equation*}
The connections forms also satisfy
$$
d\omega_\beta^\alpha-\omega_\beta^\gamma\wedge\omega_\gamma^\alpha=\frac{1}{2}R_{\beta~~\rho\sigma}^{~~\alpha}\theta^\rho\wedge\theta^{\sigma}+
\frac{1}{2}R_{\beta~~\bar\rho\bar\sigma}^{~~\alpha}\theta^{\bar\rho}\wedge\theta^{\bar\sigma}
+R_{\beta~~\rho\bar\sigma}^{~~\alpha}\theta^\rho\wedge\theta^{\bar\sigma}+R_{\beta~~\rho 0}^{~~\alpha}\theta^\rho\wedge\theta-R_{\beta~~\bar\sigma 0}^{~~\alpha}\theta^{\bar\sigma}\wedge\theta.
$$
We call $R_{\beta\bar\alpha\rho\bar\sigma}$ the pseudohermitian curvature. Contractions of the pseudohermitian curvature yield the pseudohermitian Ricci curvature $R_{\rho\bar\sigma}=R_{\alpha~~\rho\bar\sigma}^{~~\alpha}$, or $R_{\rho\bar\sigma}=h^{\alpha\bar\beta}R_{\alpha\bar\beta\rho\bar\sigma}$, and the pseudohermitian scalar curvature $R=h^{\rho\bar\sigma}R_{\rho\bar\sigma}$.

The sub-Laplacian operator in this connection can be expressed by
\begin{equation}
\Delta_b u=u^\alpha_\alpha+u^{\bar\alpha}_{\bar\alpha}
\end{equation}

If we define $\tilde{\th}=u^{\frac{2}{n}}\th$, then we have
$$\tilde{\triangle}_b f=u^{-(1+\frac{2}{n})}(u\triangle_b f+2<du,df>_{\th}),$$
where $\tilde{\triangle}_b$ is the sub-Laplacian operator with respect to the contact form $\tilde{\th}$ (see (2.4) in  \cite{Ho12} for example).   If we set
$$\tilde{u}=r^{-1}u,$$
then we have the following CR transformation law
$$(-(\ww)\tilde{\triangle}_b +\tilde{R})\tilde{u}=r^{-1-\frac{2}{n}}(-(\ww)\triangle_b +R)u.$$
If we substitute $r=u$, then we get the CR Yamabe equation \eqref{CRYE}.

If $\{W_1,\cdots,W_n\}$ is a frame for $T^{1,0}$ over some open set $U\subset M$ which is orthonormal with respect to the given pseudo-Hermitian structure on $M$, we call $\{W_1,\cdots,W_n\}$ a pseudo-Hermitian frame. $\{W_1,\cdots,W_n,\overline{W}_1,\cdots, \overline{W}_n, T\}$ forms a local frame for $\mathbb{C}TM$. Now let $U$ be a relatively compact open subset of a normal coordinate neighborhood, with contact form $\th$ and pseudo-Hermitian frame
$\{W_1,\cdots, W_n\}.$ Let $X_j={\rm{Re}} W_j$ and $X_{j+n}={\rm{Im}} W_j$. Denote $X^{\a}=X_{\a_1}\cdots X_{\a_k}$, where $\a=(\a_1,\cdots,\a_k)$. We also denote $l(\a)=k$.
Define the norm
$$\|f\|_{S_k^p(U)}=\sup_{l(\a)\leq k}\|X^{\a}f\|_{L^p(U)}.$$
The Folland-Stein space $S_k^p(U)$ is the completion of $C_0^{\infty}$ with respect to the norm $\|\cdot\|_{S_k^p(U)}$
(See \cite{FS}).  Now we use the notations in \cite{FS} as follows. Denote $H^k$ the Hilbert space $S_k^2$.  Define
$$\Gamma_{\b}(U)=\{f\in C^0(\bar{U}):|f(x)-f(y)|\leq C\r(x,y)^{\b} \},$$
with norm
$$||f||_{\Gamma_{\b}(U)}=\sup_{x\in U}|f(x)|+\sup_{x,y\in U}\frac{|f(x)-f(y)|}{\r(x,y)^{\b}}.$$
For any integer $k\geq 1$ and $k<\b<k+1$, define
$$\Gamma_{\b}(U)=\{f\in C^0(\bar{U}):X^{\a}f\in \Gamma_{\b-k}(U),l(\alpha)\leq k  \},$$
with norm
$$||f||_{\Gamma_{\b}(U)}=\sup_{x\in U}|f(x)|+\sup_{x,y\in U,l(\a)\leq k}\frac{|X^{\a}f(x)-X^{\a}f(y)|}{\r(x,y)^{\b-k}}.$$
If we fix local coordinates $(z,t)=\Theta_{\xi}$
for a fixed point $\xi\in U$, the standard H\"{o}lder space $\Lambda_{\b}(U)$ is defined for $0<\b<1$ by
$$\Lambda_{\b}(U)=\{f\in C^0(\bar{U}):|f(x)-f(y)|\leq C||x-y||^{\b} \},$$
with norm
$$||f||_{\Lambda_{\b}(U)}=\sup_{x\in U}|f(x)|+\sup_{x,y\in U,l(\a)\leq k}\frac{|X^{\a}f(x)-X^{\a}f(y)|}{||x-y||^{\b-k}}.$$
For any integer $k\geq 1$ and $k<\b<k+1$, define
$$\Lambda_{\b}(U)=\{f\in C^0(\bar{U}):(\partial/\partial x)^{\a}f\in \Lambda_{\b-k}(U),l(a)\leq k  \}.$$
Now for a compact strictly pseudo-convex psedo-Hermitian manifold $M$, choose a finite open covering $U_1, \cdots, U_m$, each $U_j$ has the properties
of $U$ above. Choose a $C^{\infty}$ partition of unity $\varphi_i$ subordinate to this covering, and define
$$S_k^p(M)=\{f\in L^1(M):\phi_j f\in S_k^p(U_j) \};$$
$$\Gamma_{\b}(M)=\{f\in C^0(M):\phi_j f\in \Gamma_{\b}(U_j) \};$$
$$\Lambda_{\b}(M)=\{f\in C^0(M):\phi_j f\in \Lambda_{\b}(U_j) \}.$$

Then we have the following  Lemma, see \cite{FS}, or Proposition 5.7 in \cite{JL-JDG}:
\begin{lemma}\label{FS1}
For each positive non-integer $\b$, each $r$, $1<r<\infty$, and each integer $k\geq 1$, there exists a constant $C$ such that for every $f\in
C_0^{\infty}(U)$,\\
(1) $||f||_{\Gamma_{\b}(U)}\leq C ||f||_{S_k^r(U)}$, where $\frac{1}{r}=\frac{k-\b}{2n+2};$\\
(2) $||f||_{\Lambda_{\b/2}}\leq ||f||_{\Gamma_{\b}(U)}$;\\
(3) $||f||_{S_2^r(U)}\leq C (||\triangle_b f||_{L^r(U)}+||f||_{L^r(U)});$\\
(4) $|f||_{\Gamma_{\b+2}(U)}\leq C(||\triangle_b f||_{\Gamma_{\b}(U)}+||f||_{\Gamma_{\b}(U)}).$\\
The constants $C$ depend only on the frame constants.
\end{lemma}
We have the following corollary immediately.
\begin{coro} \label{FS2}
Let $(M,\th)$ be a smooth, strictly pseudo-convex $2n+1$ dimensional compact CR manifold without boundary. Then there is an integer $k>0$, such that
$H^k(M)$ embeds into $C^0(M)$.
\end{coro}
\begin{proof}
This is a direct consequence of Lemma \ref{FS1} (1), and $\Gamma_{\b}(M)\subset C^0(M)$.
\end{proof}
Following CR version Sobolev Embedding Theorem was given by Jerison and Lee \cite{JL-JDG}.

\begin{prop}{\rm{(\cite{JL-JDG})}}\label{Theorem 2.1}
For $\frac{1}{s}=\frac{1}{r}-\frac{k}{2n+2}$, where $1<r<s<\infty$. Then we have
$$S_k^r(M)\subset L^s(M).$$
\end{prop}

Next we recall a CR version Poincar\'{e} inequality. In \cite{JL-JDG}, Jerison and Lee proved a Poincar\'{e} type inequality for compact, strictly pseudo-convex CR manifolds.

\begin {thm}(See \cite{JL-JDG}, Proposition 5.13) \label{p}
Let $(M,\th_0)$ be a compact, strictly pseudo-convex CR manifold, $U$ is a relatively compact open subset of a normal coordinate neighborhood of $(M,\th)$, $B_r$ is a ball of radius $r$, $B_r\subset U$. Then for any $f$ satisfying $|\nn f|\in L^q(B_r)$, $1<q<\infty$, there exits a constant $C$ independent of $f$ such that
\begin{equation}\label{p1}
\int _{B_r}|f(x)-f_{B_r}|^q \VVV\leq C r^q\int _{B_r}|\nabla_{\th_0}f|^q\VVV,
\end{equation}
where $f_{B_r}=\frac{\int _{B_r}f(x)\VVV}{\int _{B_r}\VVV}$.
\end {thm}

As a corollary of Theorem \ref{p}, we have
\begin {lemma}\label{p2}
Under the condition of Theorem \ref{p}, we have the following Poincar\'{e} type inequality:
$$\int _{B_r}|f(x)|^2 \VVV\leq C\int _{B_r}|\nabla_{\th_0}f|^2\VVV,$$
where $C$ is a positive constant independent of $f$.
\end {lemma}
\begin{proof}
We choose $v(x)$ satisfying $f(x)=v(x)-v_{B_r}$. Since $|\nabla_{\th_0}f|^2=|\nabla_{\th_0}v|^2$,  this lemma follows from Theorem \ref{p} by letting $q=2$.
\end{proof}

By the above Poincar\'{e} inequalities, we know for any $x_0\in M$, there exists a ball $B_r{(x_0)}$ such that the above Poincar\'{e} inequalities are satisfied on $B_r{(x_0)}$. Since $(M,\th_0)$ is compact, then we can obtain the following global Poincar\'{e} inequalities, which are the corollaries of Theorem \ref{p} and Lemma \ref{p2}.
\begin{coro}\label{p3}
Under the condition of Theorem \ref{p}, for any $f\in C^{\infty}(M)$, we have the following global Poincar\'{e} inequality:
\begin{equation}\label{p33}
\int _M|f(x)-\bar{f}|^2 \VVV\leq C\int_M|\nabla_{\th_0}f|^2\VVV,
\end{equation}
where $C$ is a positive constant independent of $f$, and $\bar{f}=\frac{\int _M f(x)\VVV}{\int _M\VVV}$.
\end{coro}
\begin{coro}
Under the condition of Theorem \ref{p}, for any  $f\in C^{\infty}(M)$, we have the following global Poincar\'{e} inequality:
\begin{equation}\label{p4}
\int _M|f(x)|^2 \VVV\leq C\int_M|\nabla_{\th_0}f|^2\VVV,
\end{equation}
where $C$ is a positive constant independent of $f$.
\end{coro}
Now we prove the following theorem, which is a Poincar\'{e} type inequality.
\begin{thm}\label{p5}
Let $(M,\th_0)$ be a compact, strictly pseudoconvex CR manifold. For any  $f\in C^{\infty}(M)$, we have the following global Poincar\'{e} type inequality:
$$\|\nabla_{\th_0} f\|_{L^2(M,\th_0)}\leq C\|\triangle_b f\|_{L^2(M,\th_0)},$$
for some $C>0$ independent of $f$.
\end{thm}
\begin{proof}
From Proposition 5.7(c) in \cite{JL-JDG}, we know there is a positive constant $C$ independent of $f$, such that
$$\|f\|_{S_2^2(M,\th_0)}\leq C(\|\triangle_b f\|_{L^2(M,\th_0)}+\|f\|_{L^2(M,\th_0)}).$$
Therefore we obtain
\begin{equation}\label{2nd}
\|\nn f\|_{L^2(M,\th_0)}\leq C(\|\triangle_b f\|_{L^2(M,\th_0)}+\|f\|_{L^2(M,\th_0)}).
\end{equation}

We use the contradiction argument to prove the inequality. Suppose the inequality in the theorem is not true, then there exists a sequence $\{f_j\}$ such that
$$j\|\triangle_b f_j\|_{L^2(M,\th_0)}\leq  \|\nn f_j\|_{L^2(M,\th_0)},$$
Then by \eqref{2nd}, we have
$$ \|\nn f_j\|_{L^2(M,\th_0)}\leq C(\|\triangle_b f_j\|_{L^2(M,\th_0)}+\|f_j\|_{L^2(M,\th_0)}).$$
We may require that $\|\nn f_j\|_{L^2(M,\th_0)}=1$, for any $j$. Thus, as $j$ tends to infinity, we have
$$\|\triangle_b f_j\|_{L^2(M,\th_0)}\rightarrow 0.$$
Let $u_j=f_j-\bar{f_j}$, here $\bar{f_j}=\frac{\tt f_j \VVV}{\tt \VVV}$. By \eqref{p33}, we have
$$\|f_j-\bar{f_j}\|_{L^2(M,\th_0)}\leq \|\nn f_j\|_{L^2(M,\th_0)}\leq C .$$
Then there is a subsequence of $u_j$ converges weakly in $S_2^2$, we may assume it is $u_j$ itself. Then we have $u_j\rightarrow u$ in $S_1^2$ sense for some $u$, and
$$\tt |\nn u_j|^2 \VVV=-\tt u_j\triangle_b u_j\VVV\le \|u_j\|_{L^2(M, \theta_0)}\cdot\|\triangle_b u_j\|_{L^2(M, \theta_0)}\rightarrow 0$$
as $j\rightarrow \infty$, which means $\|\nn u\|_{L^2(M,\th_0)}=0$. This contradicts the fact that $\|\nn u_j\|_{L^2(M,\th_0)}=\|\nn f_j\|_{L^2(M,\th_0)}=1$.
\end{proof}

At the end of this section, we recall some basic properties of the CR Yamabe flow \eqref {CRYF}. Under this flow, we have the following evolution equations \cite{Ho12}.
\begin {lemma}
Under the CR-Yamabe flow \eqref{CRYF}, we have
\\(1) $\frac{\partial}{\partial t}dV_{\tilde{\th}}=-(n+1)(\tilde{R}-\tilde{r})dV_{\tilde{\th}};$
\\(2) $\frac{\partial}{\partial t}u=-\frac{n}{2}(\tilde{R}-\tilde{r})u$;
\\(3) $\frac{d\tilde{r}}{dt}=-n\int_M (\tilde{R}-\tilde{r})^2dV_{\tilde{\th}};$
\\(4) $\frac{\partial}{\partial t}\tilde{R}=(n+1)\tilde{\triangle}_b\tilde{R}+(\tilde{R}-\tilde{r})\tilde{R};$
\end {lemma}
We also need the following lemmata, which were proved in \cite{Ho12} (Propositions 3.1, 3.3 and 3.4).
\begin {lemma}\label{v}
The volume of $M$ does not change under the CR Yamabe flow.
\end {lemma}
\begin {lemma}\label{r}
The function $t\mapsto \tilde{r}(t)$ is bounded from below and non-increasing under \eqref{CRYF}.
\end {lemma}


\section {Scalar flat case of the CR Yamabe flow}
By the CR Yamabe equation \eqref {CRYE}, we can reduce the CR Yamabe flow \eqref {CRYF} to the following evolution equation of the conformal factor:
\begin{equation}\label{1}
\pp u^{\rr}=\frac{(n+2)(n+1)}{n}(\ss u+\frac{n}{2n+2}\aa u^{\rr})
\end{equation}
with $u(\cdot,0)^{\qq}\th_0=\th$. We have the following lemma.
\begin {lemma}\label {3.1}
Under the condition of Theorem \ref {main}, $\aa \geq 0$ for all the time.
\end {lemma}
\begin {proof}
By the definition of $\l (M)$, we obtain
$$\l(M)=\inf\{\frac{\aa}{(\int_M u^{2+\frac{2}{n}}\VVV)^{\frac{n}{n+1}}}:u>0,u\in S^2_1(M)\}.$$
Since $\lambda(M)=0$, we therefore have $\aa \geq 0$.
\end {proof}
Then we have the following corollary:
\begin{coro}
Under the condition of Theorem \ref {main}, if $\th=\th_0$, then the Yamabe flow \eqref {CRYF1} exists for all time, and $\aa\equiv 0$, $u\equiv 1$.
\end{coro}
\begin {proof}
This is a direct consequence of Lemmata \ref{3.1} and \ref{r}.
\end {proof}
Now we prove the following theorem.
\begin {thm} \label {3.2}
Under the condition of Theorem \ref{main}, for any $T>0$, there exists a constant $C(T)$, such that $u_{\min}(0)\le u(x,t)\leq C(T)$ for $t\in [0,T]$.
\end {thm}
\begin {proof}
Since $M$ is compact, we denote $x(t)$ to be the set of points in $M$ where $u_{\min}(t)$ is obtained. Then we have
\begin{eqnarray*}
\frac{du_{\min}^{\rr}}{dt}(t) &\geq &  \inf \{\pp(u^{\rr})(x,t):x\in x(t)\}\\
 &=& \inf \{\frac{(n+2)(n+1)}{n}(\ss u+\frac{n}{2n+2}\aa u^{\rr}(t)):x\in x(t)\}\\
 &\geq & \ee \aa \um^{\rr}(t)\\
 &\geq & 0,
\end{eqnarray*}
which means
$$\um(t)\geq \um (0).$$
Similarly we get
$$\frac{d\umm ^{\rr}}{dt}(t)\leq \ee \aa \umm ^{\rr}(t)\leq \ee \aa(0) \umm ^{\rr}(t).$$
Therefore, we can obtain
$$\um(0)\leq u(x,t)\leq \umm (0)e^{\frac{n}{2}\aa(0)t}.$$
\end {proof}

\begin {thm}
Under the condition of Theorem \ref {main}, for any $T>0$, there exists a constant $C>0$ independent of $T$ such that
$$\frac{1}{C}\leq u(x,t) \leq C,$$
for any $t\in [0,T]$.
\end {thm}
\begin {proof}
First we show that the function $f(t):=(\frac{\umm(t)}{\um(t)})^{\rr}$ is non-increasing. In fact, for any $h>0$, we have
\begin{eqnarray*}
\frac{f(t+h)-f(t)}{h} &=&  \frac{1}{h}(\frac{\umm ^{\rr}(t+h)}{\um ^{\rr}(t+h)}-\frac{\umm^{\rr}(t)}{\um^{\rr}(t)})\\
&=& \frac {1}{\um^{\rr}(t+h)}\frac{\umm^{\rr}(t+h)-\umm^{\rr}(t)}{h}\\
& &-\frac{\umm^{\rr}(t)}{\um^{\rr}(t+h)\um^{\rr}(t)}\frac{\um^{\rr}(t+h)-\um^{\rr}(t)}{h}.
\end{eqnarray*}
Thus we have
\begin{eqnarray*}
& &\lim \limits_{h\rightarrow 0}\sup \frac{f(t+h)-f(t)}{h}\\
&=&\lim \limits_{h\rightarrow 0}\sup (
 \frac {1}{\um^{\rr}(t+h)}\frac{\umm^{\rr}(t+h)-\umm^{\rr}(t)}{h}\\
 & &-\frac{\umm^{\rr}(t)}{\um^{\rr}(t+h)\um^{\rr}(t)}\frac{\um^{\rr}(t+h)-\um^{\rr}(t)}{h})\\
 &\leq &\lim \limits_{h\rightarrow 0}\sup  \frac {1}{\um^{\rr}(t+h)}\frac{\umm^{\rr}(t+h)-\umm^{\rr}(t)}{h}\\
 & &-\lim \limits_{h\rightarrow 0}\inf \frac{\umm^{\rr}(t)}{\um^{\rr}(t+h)\um^{\rr}(t)}\frac{\um^{\rr}(t+h)-\um^{\rr}(t)}{h}\\
 &\leq & \frac{1}{\um^{\rr}(t)}\frac{d\umm^{\rr}}{dt}(t)-\frac{\umm ^{\rr}(t)}{(\um^{\rr}(t))^2}\frac{d\um^{\rr}}{dt}(t)\\
 &\leq &  \frac{1}{\um^{\rr}(t)}\ee \aa \umm^{\rr}(t)-\frac{\umm ^{\rr}(t)}{(\um^{\rr}(t))^2}\ee \aa \um^{\rr}(t)\\
 &=& 0.
\end{eqnarray*}
Then we get
\begin {equation} \label {2}
\frac{\umm(t)}{\um(t)}\leq \frac{\umm(0)}{\um(0)}.
\end {equation}
It has been shown in Lemma \ref{v} that the volume is invariant under the CR Yamabe flow. We therefore have
$$\vol(M,\th)=\tt u^{\ww}\VVV\geq \um ^{\ww}\vol(M,\th_0),$$
thus
$$\um(t)\leq (\frac{\vol(M,\th)}{\vol(M,\th_0)})^{\frac{n}{2n+2}}.$$
Putting these together, we obtain
$$\umm(t)\leq \frac{\umm(0)}{\um(0)}(\frac{\vol(M,\th)}{\vol(M,\th_0)})^{\frac{n}{2n+2}}.$$
\end {proof}

Once we get the $C^0$ estimate of $u(x,t)$, we may use the same argument in \cite{HSW}(page 12) to show all higher order derivatives of $u(x,t)$ are uniformly bounded on $[0,\infty)$. Then $u(t)$ converges to a smooth function $u_{\infty}$ as $t\rightarrow \infty$. Next we show that $u(t)$ converges to a smooth function $u_{\infty}$ at an exponential rate. Actually, we will show that $u_{\infty}$ is a constant. We first prove the following lemma.
\begin {lemma}
Under the condition of Theorem \ref {main}, $\aa\rightarrow 0$ as $t\rightarrow \infty$.
\end {lemma}
\begin {proof}
If $\aa\geq C>0$, for some positive constant $C$, then from the proof of Theorem \ref{3.2}, we get
\begin{eqnarray*}
\frac{du_{\min}^{\rr}}{dt}(t)&\geq & \ee \aa \um^{\rr}(t)\\
 &\geq &C\cdot\ee\cdot\um^{\rr}(t).
\end{eqnarray*}
Thus
$$\um ^{\rr}(t)\geq e^{\ee Ct}\um^{\rr}(0).$$
But this contradicts with Theorem 3.2. Therefore we have $\aa\rightarrow 0$ as $t\rightarrow \infty$.
\end {proof}
Next we show that the convergence is exponential.
\begin {lemma}
Under the condition of Theorem \ref {main}, the pseudo-Hermitian scalar curvature $\aa(t)\rightarrow 0$ exponentially as $t\rightarrow \infty$.
\end {lemma}
\begin {proof}
Since $$\pp u=(n+1)\ss u\cdot u^{-\qq}+\frac{n}{2}\aa u,$$
we have
$$\frac{1}{n+1}\pp u=\ss u\cdot u^{-\qq}+\frac{n}{2n+2}\aa u,$$
and
$$\frac{1}{n+1}\pp u \cdot\ss u=(\ss u)^2\cdot u^{-\qq}+\frac{n}{2n+2}\aa u\cdot\ss u.$$
Integrating both sides of the above equality over $M$, we have
$$\frac{1}{n+1}\tt \pp u \cdot\ss u\VVV=\tt(\ss u)^2\cdot u^{-\qq}\VVV+\frac{n}{2n+2}\aa \tt u\cdot\ss u\VVV.$$
Since
\begin{eqnarray*}
\frac{1}{n+1}\tt \pp u \cdot\ss u\VVV &=& -\frac{1}{n+1} \tt \nn u\cdot \nn (\pp u)\VVV \\
 &= & -\frac{1}{2n+2}\tt \pp |\nn u|^2 \VVV \\
 &=& -\frac{1}{2n+2}\dd \tt |\nn u|^2 \VVV,
\end{eqnarray*}
then we get
\begin{equation}\label{3}
 \frac{1}{n+1}\dd \tt |\nn u|^2 \VVV=-2\tt(\ss u)^2\cdot u^{-\qq}\VVV+\frac{n}{n+1}\aa \tt |\nn u|^2\VVV.
\end{equation}
By Theorem \ref{p5}, we have
\begin{equation}\label{4}
\parallel\nn u\parallel _{L^2(M,\th_0)}\leq C \parallel\ss u\parallel_{L^2(M,\th_0)}.
\end{equation}
Here $C$ is some positive constant independent of $u$.
By \eqref {4}, we have
\begin{eqnarray*}
\tt(\ss u)^2\cdot u^{-\qq}\VVV&\geq & \frac{1}{\umm^{\qq}}\tt(\ss u)^2\VVV \\
&\geq & C\tt |\nn u|^2 \VVV,
\end{eqnarray*}
for some positive constant $C$. Substituting this inequality into \eqref {3}, we get
$$\frac{1}{n+1}\dd \tt |\nn u|^2 \VVV\leq (\frac{n}{n+1}\aa -2C)\tt |\nn u|^2 \VVV.$$
Then for sufficiently large $t$, there exists a positive constant $A$, such that
$$\dd \log \tt |\nn u|^2 \VVV\leq (n+1)(\frac{n}{n+1}\aa -2C)\leq -A,$$
from which we get
\begin{equation}\label{5}
\aa(t)=\frac{\tt (\ww)|\nn u|^2\VVV}{\vol(M,\th)}\leq C\cdot e^{-At},
\end{equation}
for $t$ sufficiently large.
\end {proof}
From the proof of Lemma 3.3, we also get
$$ \|\nabla_{\th_0}u\|^2_{L^2(M, \th_0)}\le C\cdot e^{-At}, $$
which will be used later.

Now we prove the following theorem:
\begin {thm}
Under the condition of Theorem \ref {main}, the solution u(t) of the CR Yamabe flow \eqref{CRYF} converges to a constant at an exponential rate.
\end {thm}
\begin {proof}
Since
\begin{eqnarray*}
\dd \tt u^{\rr}\VVV &=& \tt \frac{d}{dt}(u^{\rr})\VVV\\
&=&\frac{(n+2)(n+1)}{n}\tt \ss u\VVV+\ee \aa \tt u^{\rr} \VVV\\
&=&\ee \aa\tt u^{\rr}\VVV \leq C\cdot e^{-At}\cdot\tt u^{\rr}\VVV,
\end{eqnarray*}
therefore
$\tt u^{\rr}\VVV$ is bounded from above and non-decreasing, which means
$$\lim \limits_{t\rightarrow \infty}\tt u^{\rr}(x,t)\VVV=L,$$
for some positive constant $L$.
Hence, there exists a constant $C$ such that
$$\dd \tt u^{\rr}\VVV \leq C\cdot e^{-At}.$$
Then for $t_2>t_1$, and $t_1$ sufficiently large, we have
\begin{eqnarray*}
|\tt u^{\rr}(x,t_2)\VVV-\tt u^{\rr}(x,t_1)\VVV|&=&\tt u^{\rr}(x,t_2)\VVV-\tt u^{\rr}(x,t_1)\VVV\\
&\leq&C(e^{-At_1}-e^{-At_2}).
\end{eqnarray*}
Let $t_2\rightarrow \infty$, we get
$$|\tt u^{\rr}\VVV-L|\leq C\cdot e^{-At}.$$
for $t$ sufficiently large. By Corollary \ref{p3} and H\"{o}lder inequality, we have
$$\parallel u^{\rr}-\frac{1}{V}\tt u^{\rr}\VVV\parallel^2\LL\leq C \parallel\nn u\parallel^2\LL\leq C\cdot e^{-At}.$$
Let $f=u^{\rr}-\frac{1}{V}\tt u^{\rr}\VVV$, then $\tt f \VVV=0$. We apply Theorem \ref{it1} in the Appendix below by choosing $a=\frac{1}{2}$, $p=q=r=2$, $j=k$ and $m=2k$, and use the fact that the higher order derivatives of $u$ are uniformly bounded for all $t\geq 0$, we get
$$\parallel u^{\rr}-\frac{1}{V}\tt u^{\rr}\VVV\parallel_{H^k(M, \th_0)}\leq C\cdot e^{-At}.$$
Then by Corollary \ref{FS2}, we obtain
$$| u^{\rr}-\frac{1}{V}\tt u^{\rr}\VVV| \leq C\cdot e^{-At}.$$
Let $t \rightarrow \infty$, we get $u^{\rr}\rightarrow \frac{L}{V}$ exponentially.
\end {proof}

\section{Appendix}
The Gagliardo-Nirenberg interpolation inequality is a result in the theory of Sobolev spaces that estimates the weak derivatives of a function. The estimates are in terms of $L^p$ norms of the function and its derivatives, and the inequality "interpolates" into various values of $p$ and orders of differentiation. The result is of particular importance in the theory of elliptic partial differential equations. It was proposed by Nirenberg and Gagliardo, see \cite{N}. For Riemannian case, the Gagliardo-Nirenberg type interpolation inequality was prove by Aubin (see \cite{A1}, Theorem 3.70).
Due to the lack of relevant references, we did not find the similar inequalities in CR geometry. In this section, we try to establish a Gagliardo-Nirenberg type inequality in CR geometry.

Let $(M,\th)$ be a smooth, strictly pseudoconvex $2n+1$ dimensional compact CR manifold without boundary. We choose an admissible coframe $\{\th^{\a}\}$ and dual frame $\{Z_{\a}\}$ for $T_{1,0}$. We adopt the same notations as in \cite{JL89}. Let $\alpha,\beta,\gamma,\cdots\in \{1,2\cdots, n\}$, and $a,b,c, \cdots\in \{1,2\cdots, 2n\}$, and $\bar{\a}=\a+n$. We denote $\nabla^{|j|}f$ the $j-$th covariant derivative of $f$ in the Tanaka-Webster connection in the sense
$$\parallel \nabla^{|j|}f\parallel^2=\nabla^{a_1}\nabla^{a_2}\cdots \nabla^{a_j} f\nabla_{a_1}\nabla_{a_2}\cdots \nabla_{a_j} f,$$
here $a_i\in \{1,2\cdots, 2n\}$ and $\nabla_{a_i}$ means $\nabla_{Z_{a_i}}$. From now on we denote $\parallel f \parallel_p$ be the $L^p$ norm of $f$.

By the existence of the Possion type equation $\triangle_b f=C$(see \cite{KN}). We denote $G_P(x)$ is the Green's function of the sub-Laplacian operator $\triangle_b$ which satisfies
$$\triangle_b G_P(X)=\delta_P(x)-\frac{1}{V},$$
where $V$ is the volume of $(M,\th)$, and $\delta_P(x)$ is the Dirac function at $P$. For the general case of the Green's function see \cite{CMY}.
By the definition of Dirac function, we have
\begin{equation}\label{G}
\varphi (P)=\frac{1}{V}\tt \f \VV + \tt G_P(x)\ss \f(x)\VV.
\end{equation}
We now prove the following theorem:
\begin{thm}\label{it1}
Let $(M,\th)$ be a smooth, strictly pseudoconvex $2n+1$ dimensional compact CR manifold without boundary. Let $q$, $r$ be real numbers $1\leq q,r < \infty$ and $j,m$ integers $0\leq j<m$. Then there exists a constant $K$ depending only on $n$, $m$, $j$, $q$, $r$ and $(M,\th_0)$, such that for all $f\in C^{\infty}$ with $\int f\ \VV=0$, we have:
$$\parallel \nabla^{|j|} f\parallel_p \leq K \parallel \nabla^{|m|} f \parallel_r^a \cdot \parallel f\parallel_q^{1-a}.$$
Here $\frac{1}{p}=\frac{j}{2n+2}+a(\frac{1}{r}-\frac{m}{2n+2})+(1-a)\frac{1}{q}$, for all $a$ in the interval $\frac{j}{m}\leq a <1$, for which $p$ is non-negative.
\end{thm}

We follow the idea of Aubin in \cite{A1}, we first prove the following lemma:
\begin{lemma}\label{it2}
Let $(M,\th)$ be a smooth, strictly pseudoconvex $2n+1$ dimensional compact CR manifold without boundary, and $p$, $q$ real numbers satisfying $\frac{1}{p}=\frac{1}{q}-\frac{1}{2n+2}$, $1\leq q< 2n+2$. Then there exists a constant $K$ depending only on $p$, $q$, $n$ and $(M,\th)$, for any function
$\f\in C^1(M)$ with $\tt \f \VV=0$, we have
$$\parallel\f \parallel_p\leq \parallel \nabla \f\parallel_q.$$
\end{lemma}
\begin{proof}
Since $\tt \f \VV=0$, by \eqref{G}, we have
$$\f(P)=\tt G_P(x)\ss \f (x)\VV,$$
from which we get
\begin{eqnarray*}
|\f(P)| &\leq &  \tt \parallel\nabla G_P\parallel \cdot \parallel \nabla \f\parallel \VV\\
 &=& \tt (\parallel\nabla G_P\parallel \cdot \parallel \nabla \f\parallel^q)^{\frac{1}{q}}\cdot \parallel\nabla G_P\parallel^{1-\frac{1}{q}} \VV\\
 &\leq & (\tt \parallel\nabla G_P\parallel \cdot \parallel \nabla \f\parallel^q\VV)^{\frac{1}{q}}\cdot (\tt \parallel\nabla G_P\parallel \VV)^{1-\frac{1}{q}}.
\end{eqnarray*}
Here we have used the H\"{o}lder inequality. Then we obtain
$$\parallel \f \parallel_q\leq \parallel \nabla \f \parallel_q \sup_{P\in M} \tt \parallel \nabla G_P\parallel \VV.$$
Then by Folland-Stein imbedding theorem, we obtain
$$\parallel \f\parallel_p \leq C(\parallel\nabla \f\parallel_q +\parallel \f \parallel_q)\leq K \parallel\nabla \f\parallel_q.$$
Here $K=C+C\cdot \sup_{P\in M} \tt \parallel \nabla G_P\parallel \VV.$
\end{proof}

Next, we prove the following Lemma, which is a generalized Poincar\'{e} type inequality.
\begin{lemma}\label{it3}
Let $(M,\th)$ be a smooth, strictly pseudoconvex $2n+1$ dimensional compact CR manifold without boundary, and $p$, $q$, $r$ real numbers satisfying
$1\leq q,r< \infty$, $p\geq 2$. Set $\frac{2}{p}=\frac{1}{q}+\frac{1}{r}$. Then for any functions $f\in C^{\infty}(M)$, we have:
$$\parallel \nabla f\parallel_p^2\leq (\sqrt{2n}+|p-2|)\parallel f\parallel_q\cdot \parallel \nabla^{|2|}f\parallel_r.$$
\end{lemma}
\begin{proof}
By a direct computation, we have
\begin{eqnarray*}
\nabla^a(f\parallel \nabla f\parallel^{p-2}\nabla _a f) &= &  \parallel \nabla f\parallel^p+f\parallel \nabla f\parallel^{p-2}\nabla ^a\nabla_a f\\
 &+& (p-2)\parallel \nabla f\parallel^{p-4}f\nabla_{ab}f\nabla^a f\nabla^b f.
\end{eqnarray*}
Especially, if $p=2$, we have $\parallel \nabla f\parallel_2^2=-\tt f\ss f\VV$. Then Lemma \ref{it3} is just the Poincar\'{e} type inequality we proved above. If $p>2$, we have
$$\parallel \nabla f\parallel_p^p=-\tt f\ss f \parallel \nabla f\parallel^{p-2}+(2-p)\tt \parallel \nabla f\parallel^{p-4}f\nabla_{ab}f\nabla ^a f \nabla^b f \VV.$$
Since $|\ss f|^2\leq 2n \parallel \nabla^{|2|}f\parallel^2$ and $|\nabla_{ab}f \nabla^a \nabla^b f|\leq \parallel \nabla^{|2|}f \parallel \cdot\parallel \nabla f\parallel^2$, we choose $r$ such that $\frac{1}{q}+\frac{1}{r}+\frac{p-2}{p}=1$. By H\"{o}lder inequality, we have
$$\parallel \nabla f\parallel_p^p\leq (\sqrt{2n}+|p-2|)\parallel f\parallel_q\cdot \parallel \nabla^{|2|}f\parallel_r \cdot \parallel \nabla f\parallel_p^{p-2},$$
and the desired result follows.
\end{proof}

Now we prove Theorem \ref{it1}. First we note if the two cases $j=0$, $m=1$ and $j=1$, $m=2$ are proved, the general case will be followed by induction by applying the inequality
$$\parallel \nabla \parallel\nabla ^{|l|} f\parallel\parallel\leq \parallel \nabla ^{|l+1|} f\parallel,$$
which follows from the fact that the Tanaka-Webster connection is compatible with the inner product $\langle \cdot,\cdot \rangle_{\th}$ and Cauchy-Schwarz inequality. From Lemma \ref{it2}, we have
$$\parallel f\parallel_s\leq C \parallel \nabla f\parallel_t,$$
where $\frac{1}{s}=\frac{1}{t}-\frac{1}{2n+2}>0.$

For the case $j=0$, $m=1$. By H\"{o}lder inequality, we have
$$\parallel f\parallel_p\leq \parallel f\parallel_s^a\parallel f\parallel_q^{1-a}.$$
Here $\frac{1}{p}=\frac{a}{s}+\frac{1-a}{q}$, i.e. $\frac{1}{p}-\frac{1}{q}=a(\frac{1}{s}-\frac{1}{q})$. Then we choose $t=r<2n+2$, from which we get
$$\parallel f\parallel_p\leq C \parallel \nabla f\parallel_r^a \parallel f\parallel_q^{1-a},$$
which means $\frac{1}{p}=a(\frac{1}{r}-\frac{1}{2n+2})+(1-a)\frac{1}{q}$.

If $r\geq 2n+2$, we choose $\mu$ such that $\frac{1}{ap}=\frac{1}{\mu}-\frac{1}{2n+2}$. Let $h=|f|^{\frac{1}{a}}$, we have
$$\parallel h\parallel_{ap}\leq C \parallel \nabla h\parallel_{\mu},$$
again by H\"{o}lder inequality, we have
$$\parallel f\parallel_p^{\frac{1}{a}}\leq \frac{C}{a}\parallel \|\nabla f \| \cdot |f|^{\frac{1}{a}-1}\parallel_{\mu}\leq \frac{C}{a} \parallel \nabla f\parallel_r \cdot \parallel f\parallel_q ^{\frac{1}{a}-1},$$
the desired consequence follows.

For the case $j=1$, $m=2$. If $a=\frac{j}{m}=\frac{1}{2}$, Theorem \ref{it1} is just Lemma \ref{it3}. Then for $r\geq 2n+2$, and $\frac{1}{2}<a<1$, the interpolation inequality follows from H\"{o}lder inequality. If $r\geq 2n+2$, by induction, we apply the first case to $\parallel \nabla f\parallel$ and get
$$\parallel \nabla f\parallel_p\leq C \parallel \nabla^{|2|}f\parallel_r^b \parallel \nabla f\parallel_s^{1-b},$$
where $\frac{1}{p}=\frac{1}{s}+b(\frac{1}{r}-\frac{1}{2n+2}-\frac{1}{s})>0$, $\frac{2}{s}=\frac{1}{r}+\frac{1}{q}$, and $a=\frac{1+b}{2}$. i.e.
$$\frac{1}{p}=\frac{1}{2n+2}+a(\frac{1}{r}-\frac{2}{2n+2})+(1-a)\frac{1}{q},$$
and the proof is completed.


\end{document}